\documentclass{amsart}
\usepackage{amssymb, mathtools}

\usepackage{tikz}
\usepackage{amsthm} 
\usepackage{mathrsfs} 
\usepackage{ wasysym } 
\usepackage{verbatim} 
\usepackage{amsfonts}
\usepackage{multicol}
\usepackage{graphicx}
\usepackage{enumitem}   
\usepackage{hyperref}
\usepackage{tikz-cd}
\usepackage{tabularx}
\usepackage{listings}
\usepackage[backend=biber,style=numeric,sorting=nyt,maxbibnames=99]{biblatex}
\usepackage{subfig}
\usepackage{xcolor}
\usepackage[margin=1in]{geometry}

\newcounter{theorem}
\newtheorem{theorem}[theorem]{Theorem}
\newtheorem*{theorem*}{Theorem}
\newtheorem{corollary}[theorem]{Corollary}
\newtheorem*{corollary*}{Corollary}
\newtheorem{lemma}[theorem]{Lemma}
\newtheorem{proposition}[theorem]{Proposition}
\numberwithin{theorem}{section}

\theoremstyle{definition}
\newtheorem{definition}[theorem]{Definition}
\newtheorem{example}[theorem]{Example}

\theoremstyle{remark}

\newtheorem{remark}[theorem]{Remark}

\newtheoremstyle{Question}{}{}{}{}{\color{red}\bfseries}{}{ }{}
\theoremstyle{Question}

\newcommand{\field}[1]{\ensuremath{\mathbb #1}}
\newcommand{\CC}{\field C}
\newcommand{\FF}{\field F}

\newcommand{\QQ}{\field Q}

\newcommand{\RR}{\field R}
\newcommand{\ZZ}{\field Z}

\DeclareMathOperator{\cl}{\textit{c} \ell}

\DeclareMathOperator{\End}{End}

\DeclareMathOperator{\disc}{disc}
\DeclareMathOperator{\Gal}{Gal}
\DeclareMathOperator{\Ell}{\mathcal{E}\ell \ell}

\definecolor{codegreen}{rgb}{0,0.6,0}
\definecolor{codegray}{rgb}{0.5,0.5,0.5}
\definecolor{codepurple}{rgb}{0.58,0,0.82}
\definecolor{backcolour}{rgb}{0.95,0.95,0.92}

\lstdefinestyle{mystyle}{
    backgroundcolor=\color{backcolour},   
    commentstyle=\color{codegreen},
    keywordstyle=\color{magenta},
    numberstyle=\tiny\color{codegray},
    stringstyle=\color{codepurple},
    basicstyle=\ttfamily\footnotesize,
    breakatwhitespace=false,         
    breaklines=true,                 
    captionpos=b,                    
    keepspaces=true,                 
    numbers=left,                    
    numbersep=5pt,                  
    showspaces=false,                
    showstringspaces=false,
    showtabs=false,                  
    tabsize=2
}

\title{Distribution of cycles in supersingular $\ell$-isogeny graphs}

\author{Eli Orvis}
\address{University of Colorado Boulder, Boulder, Colorado, USA}
\email{eli.orvis@colorado.edu}
\urladdr{https://euclid.colorado.edu/~wior7645/HTML/home.html}

\date{\today}

\thanks{Eli Orvis is currently supported by NSF-CAREER CNS-1652238 (PI Katherine E. Stange).}



\begin{document}

\begin{abstract}
Recent work by Arpin, Chen, Lauter, Scheidler, Stange, and Tran \cite{Arpin+22} counted the number of cycles of length $r$ in supersingular $\ell$-isogeny graphs. In this paper, we extend this work to count the number of cycles that occur along the spine. We provide formulas for both the number of such cycles, and the average number as $p \to \infty$, with $\ell$ and $r$ fixed. In particular, we show that when $r$ is not a power of $2$, cycles of length $r$ are disproportionately likely to occur along the spine. We provide experimental evidence that this result holds in the case that $r$ is a power of $2$ as well.
\end{abstract}

\maketitle

\section{Introduction}\label{SecIntroduction}

Supersingular isogeny graphs have enjoyed great interest in the past decade, primarily due to their applications to the study of modular forms \cite{Cowan22} and their use in constructing post-quantum cryptographic primitives. For distinct primes $p, \ell$, we denote the directed graph whose vertices are $\overline{\FF_p}$-isomorphism classes of supersingular elliptic curves, and whose edges are isogenies of degree $\ell$ up to post-composition by an automorphism by $\mathcal{G}_{p, \ell}$. We will use simply $\mathcal{G}_\ell$ when there is no risk of confusion. These graphs have $(\ell + 1)$-regular out-degree, $\lfloor(p-1)/12\rfloor + \epsilon$ vertices where $\epsilon \in \{0,1,2\}$, and Pizer has proven that they are Ramanujan \cite{Pizer90}. \\ 

Avoiding the existence of small cycles in $\mathcal{G}_{p,\ell}$ is advantageous in many isogeny-based cryptographic protocols. Charles, Goren, and Lauter described a method for guaranteeing that there are no small cycles in $\mathcal{G}_{p,\ell}$ when they introduced the CGL hash function, which was the first isogeny-based protocol \cite{Charles+06}. In order to assess the soundness of SIDH-based identification protocols, Ghantous, Katsumata, Pintore, and Veroni obtained upper bounds on the expected number of cycles of a given length at a random vertex of $\mathcal{G}_{p,\ell}$ \cite{Ghantous+21}. More recently, Arpin, Chen, Lauter, Scheidler, Stange, and Tran obtained a formula for the exact number of cycles of a given length in $\mathcal{G}_{p,\ell}$ in terms of class numbers of imaginary quadratic fields \cite{Arpin+22}. \\

The \emph{spine} of the graph, introduced by Arpin, Camacho-Navarro, Lauter, Lim, Newlso, Scholl, and Sot\'akov\'a in \cite{Arpin+19}, is another graph-theoretic feature of interest to isogenists. The spine is the subgraph induced by the vertices whose $j$-invariants are defined over $\mathbb{F}_p$, or alternatively, the subgraph induced by the set of vertices fixed by the Frobenius automorphism. It is known that for a density $1$ set of primes $\ell$, Frobenius is the only automorphism of the graph \cite{Mayo23}. In most cases therefore, the set of vertices of the spine is the set of fixed points of the only non-trivial symmetry of $\mathcal{G}_{p,\ell}$. We denote the spine by $\mathcal{S}_{p,\ell}$. \\

In this paper we consider the distribution of cycles within the graph, in particular with respect to the spine. We study the ratio of cycles of length $r$ that intersect the spine, for odd $r$, and show that as $p \to \infty$, the proportion of cycles of length $r$ that intersect the spine eventually depends only on the residue class of $p$ modulo a large, but computable integer. In particular, we introduce the notation $n_t$ and $n_s$ for the total number of $r$-cycles in $\mathcal{G}_{p, \ell}$, and the total number of $r$-cycles in $\mathcal{G}_{p, \ell}$ that intersect the spine, respectively, and we prove the following theorem:

\begin{theorem*}[Theorem \ref{EventualDistributionThm}]
Let $\ell$ and $r$ be fixed, with $r$ odd. Then there exists a modulus $M$ depending on $\ell$ and $r$ such that for $p \gg 0$, $n_t$ and $n_s$ depend only on $p$ modulo $M$.
\end{theorem*}

The idea behind the proof of Theorem \ref{EventualDistributionThm} is to count the number of cycles along the spine in terms of imaginary quadratic class numbers. By Theorem \ref{CycleBijectionThm} (Theorem 3.2 in \cite{Arpin+22}), every cycle of length $r$ in $\mathcal{G}_\ell$ is obtained by forgetting orientations on a subset of the oriented supersingular curves, where the orientations are from a finite set of imaginary quadratic discriminants. We then use a result of Chen and Xue \cite{ChenXue22}, which says that the $2$-torsion subgroup of the class group of $\mathcal{O}$ acts freely and transitively on the primitively $\mathcal{O}$-oriented supersingular curves to find that, for large enough $p$, the number of cycles of length $r$ intersecting the spine depends only on $p$ modulo $M$. \\

As a corollary to Theorem \ref{EventualDistributionThm}, we show that for sufficiently large $p$, the $r$-cycles in $\mathcal{G}_\ell$ either disproportionately intersect the spine, or never intersect the spine:

\begin{corollary*}[Corollary \ref{EventualDistributionCor}]
Let $\ell$ and $r$ be fixed, with $r$ odd, and $M$ be the modulus obtained in the previous Theorem. Then for any $[m] \in \ZZ / M \ZZ$, one of the following holds for all sufficiently large $p \in [m]$:
\begin{enumerate}
\item $n_t/\#V(\mathcal{G}_{p,\ell}) < n_s/\#V(\mathcal{S}_{p,\ell})$;
\item $n_s = 0$,
\end{enumerate}
where $\#V(G)$ denotes the number of vertices of $G$.
\end{corollary*}

Our final result is a formula for the average value of $n_s$ when we vary $p$ and fix $\ell$ and $r$. Specifically, Theorem \ref{AverageNsThm} provides a formula for the average of the sequence $n_{s,p_i}(r)$, where $n_{s,p_i}(r)$ is the number of cycles of length $r$ in $\mathcal{G}_{p_i, \ell}$ that intersect the spine, and $(p_i)$ is a sequence of consecutive primes:

\begin{theorem*}[Theorem \ref{AverageNsThm}]
Let $(p_i)_{i=1}^\infty$ be an increasing sequence of consecutive primes. Then 
$$\lim_{n\to \infty} \frac{1}{n}\sum_{i=1}^n n_{s,p_i}(r) = \sum_{d \mid r} \mu(d) \#d(\mathcal{I}_{\ell, n \mid \frac{r}{d}}),$$
where $d(\mathcal{I}_{\ell, n \mid \frac{r}{d}})$ is defined in Proposition \ref{DiscriminantDescriptionProp}.
\end{theorem*}

These results become particularly interesting in two contexts. First, interpreting cycles in $\mathcal{G}_{p,\ell}$ as cyclic endomorphisms, our results say that for large $p$, $\ell$-power endomorphisms of a fixed degree are disproportionately likely to occur for supersingular elliptic curves defined over $\FF_p$, as opposed to supersingular curves defined over $\FF_{p^2}$. The methods in this paper rely primarily on the structure of the supersingular isogeny graphs and the arithmetic of quaternion algebras. An independent proof from the viewpoint of arithmetic geometry would be of interest. \\

Second, we can compare these results for supersingular isogeny graphs with the behavior of random $(\ell + 1)$-regular graphs. A common heuristic is that, with the exception of the existence of an automorphism of order 2, the supersingular isogeny graphs ``behave like random $(\ell + 1)$-regular graphs.''  Our analysis provides another potential comparison between supersingular isogeny graphs and random graphs, by showing that odd length $r$-cycles in supersingular isogeny graphs do not distribute randomly throughout the graph. Instead, they are disproportionately likely to occur along the fixed vertices of the involution. It is therefore of interest to compare these results with the case of random graphs with an involution. Unfortunately, we were not able to find a construction of random graphs with an involution in the random-graph literature.

\subsection{Roadmap}

The paper is structured as follows: In Section \ref{SecPreliminaries} we provide preliminaries on supersingular elliptic curves and orientations. We also fix notation, and we recall results from the literature that will be used later on. In Section \ref{SecFpVerticesAndOrientations} we give results on the location of $\FF_p$-vertices in oriented supersingular isogeny graphs. Section \ref{SecLimitingDistribution} gives the proofs of Theorem \ref{EventualDistributionThm} and Corollary \ref{EventualDistributionCor}. We then give explicit examples of these results in Section \ref{SecExamples}. We address the expected number of cycles along the spine in Section \ref{SecExpectedValues}. Finally, we give partial results in the case of even length cycles in Section \ref{SecEvenCase}.

\subsection{Acknowledgments}

The author would like to thank Joseph Macula, James Rickards, and Katherine E. Stange for helpful conversations, and Katherine E. Stange for helpful comments on an earlier draft of this paper.
\section{Preliminaries}\label{SecPreliminaries}
 
Throughout, we will use $\mathcal{O}_D$ for the imaginary quadratic order of discriminant $D$, and simply $\mathcal{O}$ when the discriminant is not relevant or is clear from context. We will use $\cl(\mathcal{O})$ for the ideal class group of an order $\mathcal{O}$, and $h(\mathcal{O})$ for the class number of $\mathcal{O}$. The Hilbert class polynomial of an order $\mathcal{O}$ will be denoted by $\mathcal{H}_\mathcal{O}(x)$. It is well known that $\mathcal{H}_\mathcal{O}(x) \in \ZZ[x]$, and so for any prime $p$ we can consider the reduction modulo $p$. We denote this reduction by $\widetilde{\mathcal{H}_\mathcal{O}}(x)$ when there is no risk of confusion regarding the prime of reduction. \\

\subsection{Background on Supersingular Elliptic Curves}

In this section, we present some general background material on elliptic curves. Our primary reference on elliptic curves is \cite{SilvermanAEC1}. \\

An \emph{isogeny} $\phi : E_1 \to E_2$, between two elliptic curves $E_1, E_2$, is a rational map taking $0_{E_1}$ to $0_{E_2}$. The \emph{degree} of $\phi$ is the degree of the corresponding extension of function fields $[K(E_1) : \phi^*(K(E_2))]$. If the degree of $\phi$ is $\ell$ then we call $\phi$ an \emph{$\ell$-isogeny}.\\

An elliptic curve defined over $\overline{\FF_p}$ is said to be \emph{supersingular} if its endomorphism ring is non-commutative. By \cite[III.9.4]{SilvermanAEC1}, this implies that $\End(E) := \{\text{isogenies $E \to E$ defined over $\overline{\FF_p}$}\}$ is a maximal order in the unique (up to isomorphism) quaternion algebra over $\QQ$ ramified at only $p$ and $\infty$. We denote this quaternion algebra by $B_{p, \infty}$. \\

We are primarily interested in counting cycles in the $\ell$-isogeny graph, so we record here the definition of an \emph{isogeny cycle} from \cite{Arpin+22}. When we use the word ``cycle'' in reference to $\mathcal{G}_\ell$ we will always mean an isogeny cycle.

\begin{definition}\label{IsogenyCycleDef}
An \emph{isogeny cycle} is a closed walk, forgetting basepoint, in $\mathcal{G}_\ell$ containing no backtracking (no consecutive edges compose to multiplication-by-$\ell$) which is not a power of another closed walk (i.e. not equal to another closed walk repeated more than once).
\end{definition}

\begin{remark}
Note that in Definition \ref{IsogenyCycleDef}, we consider the two cycles obtained from a sequence $\{j_0, j_1, \hdots, j_{r-1}\}$ of $j$-invariants traversed in opposite directions as different cycles. We are therefore counting \emph{directed} cycles, which agrees with what was done in previous literature \cite{Arpin+22}.
\end{remark}

We will also use a refinement of supersingular elliptic curves, known as \emph{oriented supersingular curves}. For additional background on oriented supersingular curves and the oriented isogeny graph, see \cite{Arpin+22}. \\

Let $K$ be an imaginary quadratic field, $\mathcal{O}$ an order of $K$, and $E$ be a supersingular elliptic curve. Recall the following definitions from \cite{Arpin+22}:

\begin{definition}\label{Korientationdef}
A \emph{K-orientation} of $E$ is an embedding $\iota : K \to \End(E) \otimes \QQ$.
\end{definition}

\begin{definition}\label{KorientedCurvedef}
A \emph{K-oriented elliptic curve} is a pair $(E, \iota)$, where $E$ is an elliptic curve and $\iota$ is a $K$-orientation of $E$.
\end{definition}

\begin{definition}\label{Oorientationdef}
A $K$-orientation $\iota : K \to \End(E)$ is an \emph{$\mathcal{O}$-orientation} if $\iota(\mathcal{O}) \subset \End(E)$. We say that an $\mathcal{O}$-orientation is \emph{$\mathcal{O}$-primitive} if $\iota(\mathcal{O}) = \End(E) \cap \iota(K)$.
\end{definition}

\begin{definition}\label{EllFundamentalDef}
Let $\mathcal{O}$ be an imaginary quadratic order. Then $\mathcal{O}$ is \emph{$\ell$-fundamental} if $\ell$ does not divide the conductor of $\mathcal{O}$.
\end{definition}

Given an isogeny $\phi : E_1 \to E_2$, we define the $K$-orientation induced on $E_2$ from a $K$-orientation $\iota$ on $E_1$ as follows:
$$(\phi_*\iota)(\alpha) := \frac{1}{[\deg \phi]}\phi \circ \iota(\alpha) \circ \hat{\phi}.$$ A \emph{$K$-oriented isogeny} is an isogeny $\phi : (E_1, \iota_1) \to (E_2, \iota_2)$ such that $\phi_* \iota_1 = \iota_2$. Two $K$-oriented elliptic curves are \emph{$K$-isomorphic} if there exists an isomorphism $\eta : E_1 \to E_2$ such that $\eta_* \iota_1 = \iota_2$. \\

We can now introduce the oriented isogeny graph for a fixed prime $p$:

\begin{definition}\label{OrientedGraphDef}
We denote the \emph{$K$-oriented supersingular $\ell$-isogeny graph} by $\mathcal{G}_{K,\ell}$. This is the graph whose vertices are $K$-isomorphism classes of $K$-oriented supersingular elliptic curves over $\overline{\FF}_p$, and whose edges are $K$-oriented $\ell$-isogenies. We further denote by $\mathcal{G}_{\mathcal{O}, \ell}$ the subgraph of $\mathcal{G}_{K,\ell}$ induced by restricting vertices to the primitively $\mathcal{O}$-oriented supersingular curves.
\end{definition}

In Arpin, Chen, Lauter, Scheidler, Stange, and Tran \cite{Arpin+22}, the authors describe the structure of the oriented isogeny graph, as well as how cycles in $\mathcal{G}_\ell$ are obtained from cycles in $\mathcal{G}_{K,\ell}$ by forgetting orientations. In particular, we have the following results:

\begin{proposition}[\cite{Arpin+22}, Proposition 2.16]\label{OrientedVolcanoeProp} A connected component of $\mathcal{G}_{K, \ell}$, when identifying $\ell$-isogenies with their duals, has at most one cycle, passing through the the vertices that are primitively $\mathcal{O}$-oriented for some $\ell$-fundamental order $\mathcal{O}$. 
\end{proposition} 

\begin{theorem}[\cite{Arpin+22}, Theorem 3.2]\label{CycleBijectionThm}
Let $r > 2$. There is a bijection between cycles of length $r$ in $\mathcal{G}_\ell$ and directed rims of size $r$ in $\coprod_{K} \mathcal{G}_{K, \ell}$, identified up to conjugation of the orientation. We use $\coprod_{K} \mathcal{G}_{K,\ell}$ to denote the disjoint union of the $K$-oriented isogeny graphs, where $K$ ranges over all imaginary quadratic fields.
\end{theorem}

We also recall a tool used in proving the theorems cited above: the class group action on oriented curves. To define the action, we begin with the following definition from \cite{Arpin+22}:

\begin{definition}\label{OrientedActionDef}
Let $(E, \iota)$ be a primitively $\mathcal{O}$-oriented supersingular elliptic curve. Let $\mathfrak{a}$ be an integral ideal of $\mathcal{O}$ coprime to $p$. Define the intersection:

$$E[\iota(\mathfrak{a})] := \bigcap_{\alpha \in \mathfrak{a}} \ker((\iota(\alpha)).$$

This group defines an isogeny $\phi_\mathfrak{a}^{(E, \iota)} : E \to E / E[\iota(\mathfrak{a})]$. The action of $\mathfrak{a}$ on $(E, \iota)$ is defined as $\mathfrak{a} * (E, \iota) := (\phi_\mathfrak{a}^{(E, \iota)}(E), (\phi_\mathfrak{a}^{(E, \iota)})_*\iota)$. \\

We also define an action of the two-element group $\{1, \pi_p\}$ generated by the Frobenius automorphism $\pi_p$ of $\FF_{p^2}$ by 
$$\pi_p * (E, \iota) := (E^{(p)}, \iota^{(p)}),\ \pi_p * \varphi = \varphi^{(p)}$$
where $\iota^{(p)} := (\pi_p)_*(\iota)$.
\end{definition}

We can now give the primary result concerning the class group action on oriented curves:

\begin{proposition}[\cite{Arpin+22}, Proposition 2.26]\label{OrientedActionProp}
The actions of Definition \ref{OrientedActionDef} and Frobenius commute and hence give an action of $\cl(\mathcal{O}) \times \langle \pi_p \rangle$ on the primitively $\mathcal{O}$-oriented supersingular elliptic curves over $\overline{\FF_p}$. This action is transitive and its point stabilizers are either all trivial or all $\langle \pi_p \rangle$.
\end{proposition}

\subsection{Background on Quadratic Orders in $B_{p, \infty}$}

In this section, we recall several results from the literature about quadratic orders in the quaternion algebra $B_{p,\infty}$. \\

First, we have the following result of Kaneko, which says roughly that a maximal order in $B_{p,\infty}$ can only contain two quadratic discriminants $D_1, D_2$ if $D_1D_2$ is sufficiently large:

\begin{theorem}[\cite{Kaneko89}, Theorem 2]\label{KanekoPolynomialThm} Let $D_1$ and $D_2$ be two imaginary quadratic discriminants, and $p$ be a prime such that the reductions of elliptic curves with $CM$ by $D_1$ and $D_2$ are supersingular. If $D_1D_2 < 4p$, then $\widetilde{\mathcal{H}_{\mathcal{O}_{D_1}}}(x)$ and $\widetilde{\mathcal{H}_{\mathcal{O}_{D_2}}}(x)$ have no common roots in $\overline{\mathbb{F}_p}$. \end{theorem} 

In another format this gives the following:

\begin{theorem}[\cite{Kaneko89}, Theorem 2']\label{KanekoOrderThm}
Suppose that two quadratic orders $\mathcal{O}_{D_1}$ and $\mathcal{O}_{D_2}$ are optimally embedded in a maximal order of $B_{p, \infty}$ with different images. Then the inequality $D_1D_2 \geq 4p$ holds. If $\mathbb{Q}(\sqrt{D_1}) = \mathbb{Q}(\sqrt{D_2})$, this inequality can be replaced by $D_1D_2 \geq p^2$.
\end{theorem}

Chen and Xue note the following corollary, which will also be of use to us:

\begin{corollary}[\cite{ChenXue22}, Corollary 2.7]\label{ChenXueInjectivityCor} Let $\mathcal{O}$ be an imaginary quadratic order, $p$ be a prime that does not split in $\mathcal{O}$, and $\mathfrak{P}$ be a prime of the Hilbert class field of $\mathcal{O}$ lying over $p$. If $p > |\operatorname{disc}(\mathcal{O})|$, then the reduction map $r_\mathfrak{P}$ from isomorphism classes of elliptic curves over $\overline{\QQ}$ with CM by $\mathcal{O}$ to the set of isomorphism classes of supersingular elliptic curves over $\overline{\FF_p}$ is injective.
\end{corollary}

Finally, we will use the following result of Onuki, which determines when the reduction map from a CM elliptic curve $E$ to a supersingular elliptic curve $\tilde{E}$ gives an optimal embedding of $\End(E)$ into $\End(\tilde{E})$:

\begin{lemma}[\cite{Onuki20}, Lemma 3.1]\label{OnukiLemma} Let $E$ be an elliptic curve over a number field $L$ containing $K$ with $\End(E) \cong \mathcal{O} \subset K$, and $\mathfrak{p}$ a prime ideal of $L$ above $p$ such that $E$ has good reduction at $\mathfrak{p}$. Then the reduced curve $\tilde{E}$ modulo $\mathfrak{p}$ is supersingular if and only if $p$ does not split in $K$. Furthermore, let $c$ be the conductor of $\mathcal{O}$ and write $c = p^r c_0$, where $p \nmid c_0$. Then $$\End(\tilde{E}) \cap [K]_{\tilde{E}} = [\ZZ + c_0 \mathcal{O}_K]_{\tilde{E}},$$ where $[ \cdot ]_{\tilde{E}}$ reduction map between endomorphism algebras.
\end{lemma}

\subsection{Setup}

Throughout, we will use the following set of imaginary quadratic orders:

$$\mathcal{I}_{\ell,r} = \left\{\text{imaginary quadratic orders $\mathcal{O}$ : \parbox[center]{13em}{$\mathcal{O}$ is an $\ell$-fundamental order, \\ $(\ell) = \mathfrak{l}\overline{\mathfrak{l}}$ splits in $\mathcal{O}$, \\ and $[\mathfrak{l}]$ has order $r$ in $\cl(\mathcal{O})$.}} \right\}.$$

We note that $\mathcal{I}_{\ell, r}$ is a finite set, since the last two conditions require that $\ell^r$ is primitively represented by the norm form of $\mathcal{O}$. We are considering only orders with positive-definite norm forms, so there are only finitely many $\mathcal{O}$ where this condition holds. \\

We will also use the following related definitions: 
\begin{align*}
\mathrm{SS}^p & := \{\text{isomorphism classes of supersingular elliptic curves over $\overline{\FF}_p$}\}\\
d(\mathcal{I}_{\ell,r}) & := \{d : \text{$d$ is the discriminant of $\mathcal{O}$ for some $\mathcal{O} \in \mathcal{I}_{\ell,r}$}\}, \\
\Ell_{\mathcal{I}_{\ell,r}} & := \{\text{elliptic curves $E$ over $\mathbb{C}$ such that $\mathcal{O} \cong \End(E)$ for some $\mathcal{O} \in \mathcal{I}_{\ell,r}$ }\}, \\
\Ell_{\mathcal{I}_{\ell,r}}^p & := \{\text{$E \in \mathrm{SS}^p$ such that $\exists \mathcal{O} \in \mathcal{I}_{\ell, r}$ embedding optimally in $\End(E)$}\}, \\
M_{\ell, r} & := \max\{4, \max\{d_1d_2 / 4 : d_1, d_2 \in d(\mathcal{I}_{\ell, r})\}\}.
\end{align*} 

In the second set, we take curves up to isomorphism over $\CC$, and in the third up to isomorphism over $\overline{\mathbb{F}_p}$. \\

It is well-known that the $j$-invariants of elliptic curves with CM by an imaginary quadratic order $\mathcal{O}$ are the roots of $\mathcal{H}_{\mathcal{O}}(x)$. The $j$-invariants of the supersingular elliptic curves $E$ over $\overline{\mathbb{F}_p}$ with $\mathcal{O}$ optimally embedded in $\End(E)$ are the roots of the reduction of $\mathcal{H}_{\mathcal{O}}(x)$ modulo $p$. Since we are interested in curves that lie on the spine of $\mathcal{G}_\ell$, we will be interested in the $\FF_p$-roots of $\widetilde{\mathcal{H}_\mathcal{O}}$. The following result, which is Theorem 1.1 in  \cite{ChenXue22}, will be helpful in describing these roots:

\begin{theorem}[\cite{ChenXue22}, Theorem 1.1]\label{ChenXueRootCountThm}
 Let $K$ be an imaginary quadratic field and $\mathcal{O}$ be an order in $K$. Let $p$ be a prime inert in $K$ and strictly greater than $|\disc(\mathcal{O})|$, and $H_p$ be the set of $\FF_p$-roots of $\widetilde{\mathcal{H}_{\mathcal{O}}}(x)$. If $H_p$ is nonempty, then it admits a regular action by the $2$-torsion subgroup $\cl(\mathcal{O})[2] \subset \cl(\mathcal{O})$. In particular, the number of $\FF_p$-roots of $\widetilde{\mathcal{H}_\mathcal{O}}(x)$ is either zero or $|\cl(\mathcal{O})[2]|$. \\
Moreover, $H_p \neq \emptyset$ if and only if for every prime factor $q$ of $\disc(\mathcal{O})$, one of the two conditions below holds, depending on the parity of $q$:
\begin{enumerate}
\item $q \neq 2$ and the Legendre symbol $\left(\frac{-p}{q}\right) = 1$;
\item $q = 2$ and one of the following conditions holds:
\begin{enumerate}
\item $p \equiv 7 \pmod{8}$;
\item $-p + \frac{\disc(\mathcal{O})}{4} \equiv 0, 1$, or $4 \pmod{8}$;
\item $-p + \disc(\mathcal{O}) \equiv 1 \pmod{8}$.
\end{enumerate}
\end{enumerate}
\end{theorem}

We note the following corollary, which has been noted elsewhere in the literature, for example in \cite{GrossZagier84}:

\begin{corollary}\label{UniqueFpRootCor}
Let $\mathcal{O}_D$ be the maximal order in an imaginary quadratic field $K$, and $p$ be a prime that does not split in $\mathcal{O}_D$. Suppose that the class number of $\mathcal{O}_D$ is odd. Then $\mathcal{H}_{\mathcal{O}_D}(x)$ has exactly one root in $\mathbb{F}_p$.
\end{corollary}

We will omit the proof of Corollary \ref{UniqueFpRootCor}, which can be done via a case-by-case analysis from Theorem \ref{ChenXueRootCountThm}. Neither the proof nor the result will be needed in the rest of the paper. For future reference, however, we include Theorem \ref{FundGenusTheoryThm} below. This theorem is the other key ingredient in the proof of Corollary \ref{UniqueFpRootCor}.

\begin{theorem}[\cite{Cox}, Proposition 3.11]\label{FundGenusTheoryThm}
Let $D \equiv 0, 1 \pmod{4}$ be negative, and let $r$ be the number of odd primes dividing $D$. Define the number $\mu$ as follows: if $D \equiv 1 \pmod{4}$, then $\mu = r$, and if $D \equiv 0 \pmod{4}$ then $D = -4n$, where $n > 0$, and $\mu$ is determined by the following table:
\begin{center}
\begin{tabular}{ l|l } 
n & $\mu$ \\
\hline
$n \equiv 3 \pmod{4}$ & r \\
$n \equiv 1, 2 \pmod{4}$ & $r + 1$ \\
$n \equiv 4 \pmod{8}$ & $r + 1$ \\
$n \equiv 0 \pmod{8}$ & $r + 2$
\end{tabular}
\end{center}
Then the class group $C(D)$ has exactly $2^{\mu - 1}$ elements of order $\leq 2$.
\end{theorem}

\begin{remark}\label{RemarkOnInspiration} Corollary \ref{UniqueFpRootCor} inspired us to investigate the proportion of cycles of length $r$ in $\mathcal{G}_{p,\ell}$ that intersect the spine. To illustrate why, suppose that we have $\mathcal{O} \in \mathcal{I}_{\ell,r}$ of discriminant $d$, such that $p$ does not split in $\mathcal{O}$, and $p$ does not divide the conductor of $\mathcal{O}$. We have a class group action by $[\mathfrak{l}]$ on $\Ell_\mathcal{O} := \{\text{elliptic curves over $\CC$ that have $\End(E) = \mathcal{O}$}\}$. This action partitions $\Ell_\mathcal{O}$ into $\cl(\mathcal{O})/r$ orbits each of size $r$. Since $p$ does not split in $\mathcal{O}$, each curve $E \in \Ell_\mathcal{O}$ reduces to a supersingular curve modulo $p$. If we further assume that $p > d$, then the reduction map from elliptic curves with CM by $\mathcal{O}$ to supersingular curves is injective by Corollary \ref{ChenXueInjectivityCor}. The $\cl(\mathcal{O})/r$ orbits therefore reduce to distinct $r$-cycles in $\mathcal{G}_{\ell}$. Thus, if $\ell$ generates $\cl(\mathcal{O})$, and $\# \cl(\mathcal{O})$ is odd, Corollary \ref{UniqueFpRootCor} shows us that the reduced curves form a cycle of odd length that is guaranteed to intersect the spine. This provides a way of constructing cycles along the spine. \\
\end{remark}

Finally, in order to give precise counts of the number of cycles through the spine, we will need the following lemma:

\begin{lemma}\label{KanekoLemma}
Let $p > M_{\ell, r}$. Then for any prime $\mathfrak{p}$ and number field $L$ such that all curves in $\Ell_{\mathcal{I}_{\ell, r}}$ are defined over $L$ and have good reduction at $\mathfrak{p}$, the reduction map $\rho_\mathfrak{p} : \Ell_{\mathcal{I}_{\ell, r}} \to \Ell_{\mathcal{I}_{\ell, r}}^p$ is injective.
\end{lemma}

\begin{proof}
Our assumption on the size of $p$ is chosen to guarantee that for any $\tilde{E} \in \Ell_{\mathcal{I}_{\ell,r}}^p$, there is only one $\mathcal{O} \in \mathcal{I}_{\ell, r}$ that is optimally embedded in $\End(\tilde{E})$, and that there is only one such embedding of $\mathcal{O}$. The assumption will further guarantee that for any $E \in \Ell_{\mathcal{I}_{\ell, r}}$, $\rho_\mathfrak{p}(\End(E))$ is optimally embedded in $\End(\rho_\mathfrak{p}(E))$. Together, these will prove that $\rho_\mathfrak{p}$ is injective into $\Ell_{\mathcal{I}_{\ell,r}}^p$. \\ 

By Lemma \ref{OnukiLemma}, in order to show that $\rho_\mathfrak{p}$ optimally embeds endomorphism rings, it suffices to show that our assumption guarantees that $p > |d|$ for all $d \in d(\mathcal{I}_{\ell, r})$. This is elementary, since for $|d| > 4$, we have that $d^2 / 4 > |d|$. Thus $p > 4$ and $p > d_1d_2/4$ for all $d_1, d_2 \in \mathcal{I}_{\ell, r}$ guarantee that $p > |d|$ for all $d \in d(\mathcal{I}_{\ell, r})$. These conditions hold by the assumption that $p > M_{\ell,r}$. \\

To show that there is only one $\mathcal{O} \in \mathcal{I}_{\ell, r}$ that optimally embeds in $\End(\tilde{E})$ for any $\tilde{E} \in \Ell_{\mathcal{I}_{\ell, r}}^p$, we use Theorem \ref{KanekoOrderThm}. This says that if there is a maximal order in $B_{p,\infty}$ containing optimally embedded copies of both $\mathcal{O}_{d_1}$ and $\mathcal{O}_{d_2}$, then $p < \frac{d_1d_2}{4}$. Thus our assumption guarantees that $B_{p,\infty}$ contains no such order for $d_1, d_2 \in d(\mathcal{I}_{\ell, r})$. Since $\End(\tilde{E})$ is a maximal order in $B_{p, \infty}$, it follows that there is only one $\mathcal{O} \in \mathcal{I}_{\ell, r}$ that optimally embeds in $\End(\tilde{E})$.  \\ 

We now show that the reduction map $\rho_\mathfrak{p}$ is injective. Let $E_1, E_2 \in \Ell_{\mathcal{I}_{\ell, r}}$ be elliptic curves with CM by $\mathcal{O}_{d_1}$, $\mathcal{O}_{d_2}$, respectively. First, consider the case that $d_1 = d_2$. In this case, both $j(E_1)$ and $j(E_2)$ satisfy $\mathcal{H}_{\mathcal{O}_{d_1}}$. We have shown that our assumption guarantees $p > |d_1|$, and by Corollary \ref{ChenXueInjectivityCor}, this guarantees that $\mathcal{H}_{\mathcal{O}_{d_1}}$ has no repeated roots modulo $p$. Thus we see that $E_1$ and $E_2$ cannot reduce to supersingular curves with the same $j$-invariant, and so $\rho_\mathfrak{p}(E_1) \neq \rho_\mathfrak{p}(E_2)$. \\

For the second case, suppose that $d_1 \neq d_2$, and suppose $\rho_\mathfrak{p}(E_1) = \rho_\mathfrak{p}(E_2) = E$. Then since $\rho_\mathfrak{p}$ optimally embeds the endomorphism rings of $E_1, E_2$ into $\End(E)$, we have that $\End(E)$ is a maximal order of $B_{p, \infty}$ with optimally embedded copies of $\mathcal{O}_{d_1}$ and $\mathcal{O}_{d_2}$. This contradicts Theorem \ref{KanekoOrderThm}, and so the result is proven.
\end{proof}

\begin{remark}\label{KanekoLemmaRmk}
We remark that the results and proof of Lemma \ref{KanekoLemma} continue to hold for a finite set $\{r_1, \hdots, r_n\}$ of $r$-values, provided that $M_{\ell, r}$ is replaced by the maximum of the $M_{\ell, r_i}$. This will be used in Section \ref{SecEvenCase} when considering the case of even cycles.
\end{remark}

\section{$\FF_p$-vertices on oriented isogeny cycles and the even class number case}\label{SecFpVerticesAndOrientations}

In this section we will count the number $\FF_p$ vertices on oriented isogeny cycles. This will allow us to extend the ideas from Remark \ref{RemarkOnInspiration} to the case of even class number. Our strategy is to apply results of Chen and Xue \cite{ChenXue22} about the class group action on $\FF_p$-roots of the Hilbert class polynomial to oriented supersingular curves. We will conclude that for $p > d$, the number of $r$-cycles obtained from primitive $\mathcal{O}$-orientations that intersect the spine is either $0$ or $2^{\mu}$, where $\mu$ defined as in Theorem \ref{FundGenusTheoryThm}. \\

Our first theorem in this section describes the action of the class group on $\FF_p$-vertices of oriented isogeny cycles:

\begin{theorem}\label{FpVertexOrientationSymmetry}
Let $p > |\disc(\mathcal{O})|$. Then the restriction of the action of $\cl(\mathcal{O}) \times \langle \pi_p \rangle$ to $\cl(\mathcal{O})[2] \times \langle \pi_p \rangle$ acts transitively on the $\FF_p$-vertices. 
\end{theorem} 

\begin{proof}
By Proposition \ref{OrientedActionProp}, the action of $\cl(\mathcal{O}) \times \langle \pi_p \rangle$ is transitive on the vertices of $\mathcal{G}_{\mathcal{O}, \ell}$, so we need only show that for $k \in \{1,2\}$ and $(E, \iota)$ with $j(E) \in \FF_p$, $(E^{(\mathfrak{a}, \pi_p^k)}, \iota^{(\mathfrak{a}, \pi_p^k)})$ has $j(E^{(\mathfrak{a}, \pi_p^k)}) \in \FF_p$ if and only if $[\mathfrak{a}] \in \cl(\mathcal{O})[2]$. \\

Suppose that $j(E) \in \FF_p$. Then $j(E^{(p)}) = j(E)$. Since $E^{(\mathfrak{a}, \pi_p)} = (E^{(p)})^{\mathfrak{a}}$, we are reduced to showing that $j(E^{\mathfrak{a}}) \in \FF_p$ if and only if $[\mathfrak{a}] \in \cl(\mathcal{O})[2]$. This follows from Theorem \ref{ChenXueRootCountThm} under the assumption that $p > |\disc(\mathcal{O})|$.
\end{proof}

As a corollary, we see that any two cycles in $\mathcal{G}_{\mathcal{O},\ell}$ containing at least one $\FF_p$-vertex must contain the same number of $\FF_p$-vertices:

\begin{corollary}\label{FpVertexOrientationSymmetryCorollary}
Let $C_1$, $C_2$ be cycles in $\mathcal{G}_{\mathcal{O}, \ell}$ such that both contain an $\FF_p$-vertex. If $p > |\disc(\mathcal{O})|$, then $C_1$ and $C_2$ have the same number of $\FF_p$-vertices.
\end{corollary}

\begin{proof}
By Proposition \ref{OrientedVolcanoeProp}, we know that each connected component of $\mathcal{G}_{K, \ell}$ has at most one cycle. By definition \ref{OrientedGraphDef}, $\mathcal{G}_{\mathcal{O}, \ell}$ is the subgraph of $\mathcal{O}$-rims, so we see that $\mathcal{G}_{\mathcal{O}, \ell}$ is a collection of disjoint cycles, possibly containing isolated vertices. Since the action of $\cl(\mathcal{O}) \times \langle \pi_p \rangle$ preserves adjacency, we get that $\cl(\mathcal{O}) \times \langle \pi_p \rangle$ takes cycles of $\mathcal{G}_{\mathcal{O}, \ell}$ to cycles. \\

Let $C_1, C_2$ be cycles in $\mathcal{G}_{\mathcal{O}, \ell}$, and let $C_i^p = \{\text{vertices $(E, \iota) \in C_i$ such that $j(E) \in \FF_p$}\}$. Let $v_1, v_2$ be elements of $C_1^p, C_2^p$, respectively. We will show that $\# C_1^p = \# C_2^p$. By Theorem \ref{FpVertexOrientationSymmetry}, there exists an element $g \in \cl(\mathcal{O})[2] \times \langle \pi_p \rangle$ taking $v_1$ to $v_2$, and therefore necessarily taking $C_1$ to $C_2$. Further, $g$ acts injectively, and takes $C_1^p$ to $C_2^p$. We see therefore that $\# C_2^p \geq \# C_1^p$. Replacing $g$ by $g^{-1}$ gives the opposite inequality, which proves that $\# C_1^p = \# C_2^p$.
\end{proof}

Finally, we will show that for sufficiently large $p$ and odd $r$, the number of $\FF_p$-vertices on a cycle in $\mathcal{G}_{\mathcal{O}, \ell}$ is either $0$ or $1$, for any $\mathcal{O} \in \mathcal{I}_{\ell, r}$. \\

We begin by making explicit a result from Arpin, Chen, Lauter, Scheidler, Stange, and Tran \cite{Arpin+22}:

\begin{lemma}\label{CyclesInGOdellLengthrLemma}
Let $r$ be odd, and $\mathcal{O} \in \mathcal{I}_{\ell, r}$. Then every cycle in $\mathcal{G}_{\mathcal{O}, \ell}$ has length $r$.
\end{lemma}

\begin{proof}
We will show that every edge in $\mathcal{G}_{\mathcal{O}, \ell}$ comes from the action of $[\mathfrak{l}]$ or $[\overline{\mathfrak{l}}]$ via a counting argument. \\

Since $\ell$ splits in $\mathcal{O}$ and does not divide the conductor of $\mathcal{O}$, we have by \cite[Proposition 2.15]{Arpin+22} that there are 2 horizontal $\ell$-isogenies out of each $(E, \iota) \in \mathcal{G}_{\mathcal{O}, \ell}$. We now consider the isogeny $\phi_\mathfrak{l} : (E, \iota) \to [\mathfrak{l}] \star (E, \iota)$ whose kernel is $E[\iota(\mathfrak{l})]$. Since $N(\mathfrak{l}) = \ell$, $\phi_{\mathfrak{l}}$ is an $\ell$-isogeny, and by definition of the orientation on $[\ell] \star (E, \iota)$, $\phi_{\mathfrak{l}}$ is a $K$-oriented $\ell$-isogeny. Thus $\phi_{\mathfrak{l}}$ is one of the horizontal edges out of $(E, \iota)$, and by the same argument we get an isogeny $\phi_{\overline{\mathfrak{l}}}$ that gives the other horizontal edge. \\

The previous paragraph allows us to explicitly describe the action of $[\mathfrak{l}]$ on $\mathcal{G}_{\mathcal{O}, \ell}$. Since each vertex in $\mathcal{G}_{\mathcal{O}, \ell}$ has two horizontal $\ell$-isogenies out of it, we see that there are no isolated vertices in $\mathcal{G}_{\mathcal{O}, \ell}$, and so $\mathcal{G}_{\mathcal{O}, \ell}$ is a disjoint union of cyclic subgraphs. Further, $[\mathfrak{l}]$ acts on $\mathcal{G}_{\mathcal{O}, \ell}$ by cyclicly permuting the vertices in each cycle. The length of each cycle is therefore the minimum $s$ such that $[\mathfrak{l}^s]$ acts as the identity on each vertex. By Proposition \ref{OrientedActionProp}, the point stabilizers of the vertices in $\mathcal{G}_{\mathcal{O}, \ell}$ are all either trivial or $\langle \pi_p \rangle$, so we see that we must have $[\mathfrak{l}^s]$ is the identity in $\cl(\mathcal{O})$. Thus the length of each cycle is equal to the order of $[\mathfrak{l}]$ in the class group, namely $r$. \\
\end{proof} 

We aim to prove that for sufficiently large primes $p$, every odd cycle intersecting the spine contains at most one $\FF_p$-vertex. Our first step is to prove Lemma \ref{FrobActionOnCyclesLemma} below, which says that for large enough $p$, a cycle $C$ of length $r$ through the spine is either fixed by action of Frobenius, or its orientation is reversed. To make this precise, we recall the following definition from \cite{Arpin+22}:

\begin{definition}\label{BackwardWalkDef}
Given a cycle $C := \phi_1, \phi_2, \hdots, \phi_n$ in $\mathcal{G}_\ell$, we define the \emph{opposite cycle} to be the cycle $\hat{C}$ given by the dual isogenies in the reverse order. Graph theoretically, the opposite cycle is a cycle that traverses the same vertices in the opposite order. 
\end{definition}  

\begin{remark}\label{SafeArbitraryAssignmentRmk}
We remark that Definition \ref{BackwardWalkDef} requires special care at the vertices $j = 0$ and $j = 1728$. The extra automorphisms at these vertices make it possible for more than one outgoing edge to correspond to the same incoming edge when taking the dual. This subtelty is handled in \cite{Arpin+22} by defining a \emph{safe arbitrary assignment} (see Definition 3.13 in \cite{Arpin+22}). For our purposes, it is sufficient to know that we can make suitable choices so that every cycle has a unique corresponding opposite cycle. Further, if $\theta$ is the endomorphism corresponding to the composition of the isogenies in $C$, then $\pm \hat{\theta}$ is the endomorphism corresponding to the composition of the isogenies in $\hat{C}$. The interested reader can consult section 3 of \cite{Arpin+22} for details.
\end{remark}

\begin{lemma}\label{FrobActionOnCyclesLemma}
Let $p > M_{\ell, r}$, $\pi_p$ be Frobenius, and $C$ be a cycle of length $r$ in $\mathcal{G}_\ell$. Suppose that $C$ contains an $\FF_p$-vertex $v$. Then we have that $\pi_p(C) \in \{C, \hat{C}\}$. 
\end{lemma}

\begin{proof}
Since $v$ is defined over $\FF_p$, we have that $\pi_p(v) = v$. Further, we know that $\pi_p$ is a graph isomorphism, and so $\pi_p$ takes $C$ to a cycle of length $r$ through $v$. By Theorem \ref{CycleBijectionThm}, every such cycle corresponds to an optimal embedding of some $\mathcal{O} \in \mathcal{I}_{\ell, r}$ into $\End(v)$. Lemma \ref{KanekoLemma} tells us that, up to conjugation, there is only one such optimal embedding. Thus there are only two cycles of length $r$ through $C$, and they correspond to dual endomorphisms.   The two cycles must therefore be opposite, which proves the $\pi_p(C) \in \{C, \hat{C}\}$.
\end{proof}

We can now prove that each $r$-cycle has at most one $\FF_p$-vertex. 

\begin{proposition}\label{FpvertexOrientationSymmetryOddLengthProp}
Let $r$ be odd, and $p > M_{\ell, r}$. Then the number of $\FF_p$-vertices on each $r$-cycle in $\mathcal{G}_{\mathcal{O}, \ell}$ is either $0$ or $1$. The same result also holds for $r$-cycles in $\mathcal{G}_\ell$.
\end{proposition}

\begin{proof}
Let $C$ be a cycle containing at least one $\FF_p$-vertex in $\mathcal{G}_{\mathcal{O}, \ell}$, $\tilde{C}$ be the cycle in $\mathcal{G}_\ell$ obtained by forgetting orientations, and $\tilde{C}^p$ be the set of $\FF_p$-$j$-invariants appearing as vertices in $\tilde{C}$. We define $C^p$ similarly. By Lemma \ref{CyclesInGOdellLengthrLemma} we know that the length of $C$ is $r$. We will show that $\# C^p$ is odd. \\

By Lemma \ref{FrobActionOnCyclesLemma}, we have that Frobenius fixes the vertex set of $\tilde{C}$, and therefore also fixes the set of $j$-invariants that appear in vertices of $C$. There is therefore an even number of $\FF_{p^2}$-vertices in $C$ and in $\tilde{C}$. Since there is an odd number of total vertices in $C$, we see that there is an odd number of $\FF_p$-vertices in $C$. \\ 

We now show that our assumption on $p$ guarantees there are a $2$-power number of $\FF_p$-vertices in $\mathcal{G}_{\mathcal{O}, \ell}$. By Theorem \ref{ChenXueRootCountThm}, the number of $\FF_p$-$j$-invariants obtained by forgetting orientations from $\mathcal{G}_{\mathcal{O}, \ell}$ is $2^{\mu - 1}$. We can therefore count $\FF_p$-vertices in $\mathcal{G}_{\mathcal{O},\ell}$ by counting how many distinct orientations we obtain on each $j$-invariant. By the argument in Lemma \ref{KanekoLemma}, our assumption that $p > M_{\ell, r}$ implies that $p > |d|$. Thus $p$ cannot ramify in $\mathcal{O}$, and so we must have that $p$ is inert in $\mathcal{O}$. By \cite[Proposition 4.2]{Arpin+22}, this implies that for any orientation $\iota$ of $E$, $(E, \iota)$ and $(E, \overline{\iota})$ are non-isomorphic as oriented supersingular curves. As argued in Lemma \ref{KanekoLemma}, the assumption that $p > M_{\ell, r}$ also guarantees that there is only one embedding of $\mathcal{O}$ into $\End(E)$, so we see that $j(E)$ appears exactly twice in $\mathcal{G}_{\mathcal{O}, \ell}$. Thus the number of $\FF_p$-vertices in $\mathcal{G}_{\mathcal{O}, \ell}$ is $2(2^{\mu - 1}) = 2^\mu$. \\

By Corollary \ref{FpVertexOrientationSymmetryCorollary}, each cycle with at least one $\FF_p$-vertex in $\mathcal{G}_{\mathcal{O}, \ell}$ has the same number of $\FF_p$-vertices. This number is odd by the argument in the second paragraph. The total number of $\FF_p$-vertices in $\mathcal{G}_{\mathcal{O}, \ell}$ is $2^\mu$ by the arguments above. Letting $2k + 1$ be the number $\FF_p$-vertices on any cycle in $\mathcal{G}_{\mathcal{O}, \ell}$ with at least one such vertex, we see that $\#\{\text{cycles in $\mathcal{G}_{\mathcal{O}, \ell}$ with an $\FF_p$-vertex}\}(2k + 1) = 2^\mu$. It follows that $k = 0$, and so there is either $0$ or $1$ such vertices on each cycle. Thus $\# C^p \in \{0,1\}$. \\

Finally, we note that this implies $\# \tilde{C} \in \{0,1\}$ as well, since the sets of $j$-invariants appearing as vertices in $C$ and $\tilde{C}$ are the same.   
\end{proof}

\section{Limiting Distribution of cycles on and off the spine}\label{SecLimitingDistribution}

By the results of the previous sections, we know that for sufficiently large primes $p$, all of the following hold:

\begin{enumerate}
\item $\mathcal{I}_{\ell, r}$ is a finite list of imaginary quadratic orders $\mathcal{O}$ such that every cycle of length $r$ in $\mathcal{G}_{\ell}$ consists of vertices whose $j$-invariants satisfy the reduction modulo $p$ of the Hilbert class polynomial for some $\mathcal{O} \in \mathcal{I}_{\ell, r}$.
 
\item For each order $\mathcal{O} \in \mathcal{I}_{\ell, r}$, the number of cycles obtained from the reductions of curves with CM by $\mathcal{O}$ is either $0$ if $p$ splits in $\mathcal{O}$, or $2h(\mathcal{O})/r$ otherwise.

\item The number of cycles obtained from the reduction of curves with CM by $\mathcal{O}$ that lie along the spine is either $0$, or $2^{\mu}$. By Theorem \ref{ChenXueRootCountThm}, whether there are $0$ or $2^\mu$ cycles along the spine depends only on the discriminant, and the value of $p$ modulo $8$.
\end{enumerate}

Together, these allow us to prove the first Theorem of the introduction:

\begin{theorem}\label{EventualDistributionThm}
Let $\ell$ and $r$ be fixed, with $r$ odd. Then there exists a modulus $M$ depending on $\ell$ and $r$ such that for $p \gg 0$, $n_t$ and $n_s$ depend only on $p$ modulo $M$.
\end{theorem}

\begin{proof}
We will see that for sufficiently large $p$, the number of $r$-cycles contributed by each order $\mathcal{O} \in \mathcal{I}_{\ell, r}$, both in total, and along the spine, depends only on congruence conditions on $p$. \\

Specifically, $\mathcal{O}$ contributes no $r$-cycles if $p$ splits in $\mathcal{O}$, which is equivalent to $d$ being a quadratic residue modulo $p$. By quadratic reciprocity this is equivalent to certain congruence conditions on $p$ modulo $d$ or $4d$. On the other hand, if $p$ does not split in $\mathcal{O}$, then $\mathcal{O}$ gives $2^{\mu}$ or $0$ cycles along the spine according to $p$ modulo $8$ and congruence conditions on $p$ modulo $q$ for $q \mid d$, as in Theorem \ref{ChenXueRootCountThm}, and gives $2h(\mathcal{O})/r$ cycles in total. \\

For each discriminant $d \in d(\mathcal{I}_{\ell, r})$, we therefore have three sets of congruences $S_{d,1}, S_{d,2}, S_{d,3} \subseteq \ZZ / 8d \ZZ$ with the following properties: if $[p] \in S_{d,1}$ then there are no cycles in $\mathcal{G}_\ell$ obtained from reduction of $\mathcal{O}$; if $[p] \in S_{d,2}$ then there are $2h(\mathcal{O}) / r$ cycles in $\mathcal{G}_\ell$ obtained from reduction of $\mathcal{O}$, none of which lie along the spine; and if $[p] \in S_{d,3}$, then there $2h(\mathcal{O}) / r$ cycles in $\mathcal{G}_\ell$ obtained from reduction of $\mathcal{O}$, with $2^{\mu}$ of them lying along the spine. For $p > M_{\ell, r}$, we have that the cycles obtained by $\mathcal{O}_1$ and $\mathcal{O}_2 \neq \mathcal{O}_1$ are distinct by Lemma \ref{KanekoLemma}, so we can obtain the exact count of cycles both in total, and along the spine from the value of $p$ modulo $\operatorname{lcm}(8, d_1, \hdots, d_k)$, where the $\{d_1, \hdots, d_k\} = d(\mathcal{I}_{\ell, r})$. 
\end{proof}

We are now in a position to prove the following Corollary as well:

\begin{corollary}\label{EventualDistributionCor}
Let $\ell$ and $r$ be fixed, with $r$ odd, and $M$ be the modulus obtained in the previous Theorem. Then for any $[m] \in \ZZ / M \ZZ$, one of the following holds for all sufficiently large $p \in [m]$:
\begin{enumerate}
\item $n_t/\#V(\mathcal{G}_{p,\ell}) < n_s/\#V(\mathcal{S}_{p,\ell})$;
\item $n_s = 0$.
\end{enumerate}
\end{corollary}

\begin{proof}
In \cite{Arpin+19}, the authors give the number of spine vertices as a constant multiple of either $h(\mathcal{O}_{-4p})$ or $h(\mathcal{O}_{-p})$, depending on $p$ modulo $8$. It is a standard bound in analytic number theory that $h(\mathcal{O}_d) \leq C\sqrt{|d|}\log(|d|)$ for imaginary quadratic discriminants $d$, and a constant $C$ \cite[Equation 8.11]{JacobsonWilliams09}. Letting $s$ be the number of spine vertices we therefore have that $s = O(\sqrt{p}\log{p})$. On the other hand, the number of vertices in $\mathcal{G}_{p, \ell}$ is $O(p)$. \\

Let $n_t$ and $n_s$ be as in the Corollary statement. Using the notation from the proof of Theorem \ref{EventualDistributionThm}, we see that if $[m] \in S_{d,1}$ or $[m] \in S_{d,2}$ for all $d \in d(\mathcal{I}_{\ell, r})$, then $n_s = 0$ for all sufficiently large $p \in [m]$. Otherwise, we have that $n_s, n_t > 0$. \\

Proposition \ref{FpvertexOrientationSymmetryOddLengthProp}
says that each cycle along the spine contains $1$ spine vertex. Further, for $p > M_{\ell, r}$, all of the $r$-cycles in $\mathcal{G}_{p,\ell}$ are disjoint, unless the two cycles are the same up to direction of traversal. We therefore see that, for $p$ sufficiently large, each spine vertex contained in an $r$-cycle is contained in exactly two $r$-cycles. Thus there are $n_s/2$ spine vertices that are part of an $r$-cycle. Each $r$-cycle contains $r$-total vertices, and as argued above, for sufficiently large $p$, the $r$-cycles are all disjoint, except for a cycle and its opposite. Thus there are $rn_t/2$ vertices contained in $r$-cycles. We therefore have that for $p \equiv m \pmod{M}$, the proportion of spine vertices contained in an $r$-cycle is $n_s / O(\sqrt{p}\log(p))$, while the proportion of all vertices contained in an $r$-cycles is $n_t / O(p)$. Since $O(p)$ is asymptotically greater than $O(\sqrt{p}\log(p))$, we see that the proportion of spine vertices with an $r$-cycle is eventually greater than the proportion of all vertices. 
\end{proof}

\section{Explicit examples}\label{SecExamples}

In this section, we present explicit examples illustrating the key ideas in Theorem \ref{EventualDistributionThm} and Corollary \ref{EventualDistributionCor}. We consider the distribution of 3 cycles in the 3-isogeny graph. \\

\begin{table}
\begin{tabularx}{0.95\textwidth}{|l|l|l|X|}\hline
discriminant & class number & $\left(\frac{d}{4643}\right)$ & $\mathbb{F}_{4643^2}$ $j$-invariants \\ \hline
-23 & 3 & -1 & $173, 1283 z_{2} + 3319, 3360 z_{2} + 937$ \\ \hline
-44 & 3 & 1 &   \\ \hline
-59 & 3 & 1 & $1896, 1161, 851$ \\ \hline
-83 & 3 & 1 & $3690, 2630, 1486$ \\ \hline
-92 & 3 & -1 & $4537, 732 z_{2} + 4024, 3911 z_{2} + 1326$ \\ \hline
-104 & 6 & -1 & $2549 z_{2} + 3241, 3268 z_{2} + 3162, 3037 z_{2} + 3085, 2094 z_{2} + 2967, 1606 z_{2} + 2560, 1375 z_{2} + 73$ \\ \hline
-107 & 3 & 1 & $3870, 1520, 637$ \\ \hline
\end{tabularx}
\caption{\label{ThreeThreeTable}Discriminants in $\mathcal{I}_{3,3}$, and corresponding data for the prime $p = 4643$. The variable $z_2$ is a root of the polynomial $x^2 - 9x - 4638 \in \FF_{4643}[x]$.}
\end{table}

Referring to Table \ref{ThreeThreeTable}, we see that $\mathcal{I}_{3,3} =  \{-23, -44, -59, -83, -92, -104, -107\}$. For $p > \frac{(-104)(-107)}{4}$, we have that the reductions of the curves with CM by these orders are all distinct. In these cases we can count the number of cycles in $\mathcal{G}_\ell$ that intersect $\mathcal{S}_{\ell}$ by counting the number resulting from each order separately. We will assume that $p$ is sufficiently large for the rest of this example. \\

Consider the order $\mathcal{O}$ of discriminant $-23$. Reductions of curves with CM by $\mathcal{O}$ are supersingular if and only if $\left(\frac{-23}{p}\right) = -1$. By quadratic reciprocity, this occurs if and only if $$p \equiv 3,
 5,
 17,
 21,
 23,
 27,
 31,
 33,
 35,
 37,
 39,
 45,
 47,
 53,
 55,
 57,
 59,
 61,
 65,
 71,
 75,
 87,
 89 \pmod{92}.$$ If $p$ is in one of these residue classes, then the reductions of the curves with CM by $\mathcal{O}$ form two directed $3$-cycles in the $3$-isogeny graph. Further, $\left(\frac{-p}{23}\right) = 1$ in each of these cases, so Theorem \ref{ChenXueRootCountThm} tells us that these $3$-cycles intersects the spine. \\
 
Next, we analyze the case of $d = -104$, where more complicated scenarios can occur. We find again by quadratic reciprocity, that the curves with CM by $\mathcal{O}_{-104}$ reduce to supersingular curves if and only if 
\small \begin{equation}\label{104SSresidues}
p \equiv  11,
 19,
 23,
 29,
 33,
 41,
 53,
 55,
 57,
 59,
 61,
 67,
 69,
 73,
 77,
 79,
 83,
 87,
 89,
 95,
 97,
 99,
 101,
 103 \pmod{104}.
\end{equation}  \normalsize
The class group of $\mathcal{O}_{-104}$ is isomorphic to $\ZZ / 6 \ZZ$, so that for $p > 104$, there are either $0$ or $2$ $\FF_p$-roots of the Hilbert Class Polynomial. By Theorem \ref{ChenXueRootCountThm}, there are $2$ $\FF_p$-roots as long as 
 \begin{enumerate}
 \item $\left(\frac{-p}{13}\right) = 1$, and 
 \item at least one of the following holds:
 \begin{enumerate}
 \item $p \equiv 7 \pmod{8}$,
 \item $-p - 26 \equiv 0, 1, \text{or } 4 \pmod{8}$, or,
 \item $-p - 104 \equiv 1 \pmod{8}$.
 \end{enumerate}
 \end{enumerate}
Simplifying these conditions shows that $p$ must be in the following residue classes: \begin{equation}\label{104Fpresidues} p \equiv 23, 29, 53, 55, 61, 69, 77, 79, 87, 95, 101, 103 \pmod{104}. \end{equation} Thus the discriminant $\mathcal{O}_{-104}$ produces $4$ directed $3$-cycles in $\mathcal{G}_{p,3}$ if and only if $p$ is in one of the residues listed in \eqref{104SSresidues}, and these four $3$-cycles lie along the spine if and only if $p$ is in one of the residues listed in \eqref{104Fpresidues}. \\

The analyses of the remaining discriminants are similar. For our purposes, it suffices to note that in each case, we can give an explicit list of residues modulo some $m_i$ where the order produces supersingular curves, and further, when the class number is even we can give an explicit subset of these residues where there are $2^{\mu}$ reduced curves along the spine. Considering each residue class in $(\ZZ / \operatorname{lcm}(m_i) \ZZ)^\times$, we can then count the number of $3$-cycles in the graph $\mathcal{G}_{p,3}$, as well as the number of $3$-cycles along the spine. By explicit computation, we see that in this case, $\operatorname{lcm}(m_i) = 13,786,935,448$, so that for sufficiently large primes $p$, the number of $3$-cycles in total and along the spine in the $3$-isogeny graph depends only on $p$ modulo $13,786,935,448$. \\

For example, consider $p = 4643$. Table \ref{ThreeThreeTable} shows that in this case only the orders $\mathcal{O}_{-23}$, $\mathcal{O}_{-92}$, and $\mathcal{O}_{-104}$ produce supersingular $j$-invariants, and further we have that none of the $3$-cycles coming from $\mathcal{O}_{-104}$ lie along the spine. We therefore have that $1/2$ of the $3$-cycles lie along the spine and $1/2$ lie off of the spine. Since this ratio depends only on $p$ modulo $13,786,935,448$, we see that for large enough $p$ equivalent to $4643$ modulo $13,786,935,448$, $1/2$ of the three cycles in $\mathcal{G}_{p,3}$ lie along spine, while the proportion of vertices that are on the spine approaches $0$.

\section{Expected values for $n_s$}\label{SecExpectedValues}

In this section, we give an explicit formula for $n_s$, as well as for the expected value of $n_s$ as $p \to \infty$. These formulae are given in terms of imaginary quadratic class numbers and discriminants, but without conditions on the order or splitting behavior of $\ell$ in the class group. \\

We first recall a convenient description for the set $\cup_{n \mid r} d(\mathcal{I}_{\ell,n})$ given in the proof of Theorem 7.4 in \cite{Arpin+22}:

\begin{proposition}\label{DiscriminantDescriptionProp}
The set of imaginary quadratic discriminants $\Delta$ where $\ell\mathcal{O}_\Delta$ splits, the primes above $\ell$ have order dividing $r$, and the conductor is not divisible by $\ell$, is given by $$\left\{\frac{x^2 - 4\ell^r}{f^2} : 0 < x < 2\ell^{r/2},\ x \not\equiv 0 \pmod{\ell},\ f^2 \mid x^2 - 4\ell^r, \text{ and } \frac{x^2 - 4\ell^r}{f^2} \equiv 0,1 \pmod{4} \right\}.$$ We will denote this by $d(\mathcal{I}_{\ell, n \mid r})$.
\end{proposition}

\begin{proof}
Let $p$, $\ell$, and $r$ be fixed. Theorem 7.4 of \cite{Arpin+22} shows that the set of discriminants corresponding to cycles whose length divides $r$ in $\mathcal{G}_{p, \ell}$ is given by the set of $\frac{x^2 - 4\ell^r}{f^2}$ such that all of the conditions above hold, but also $x^2 - 4\ell^r$ is not a quadratic residue modulo $p$ and has valuation at most $1$ at $p$. The set above simply removes the conditions on $p$.
\end{proof}

We now aim to compute the average number of spine cycles for large $p$. We first use M\"obius inversion to find a formula for this count for a specific $p$. Let $n_{s,p}(r)$ be the number of spine cycles of length $r$ in $\mathcal{G}_{\ell}$.

\begin{theorem}\label{ExplicitNsCountThm}
Let $p, \ell$, and $r$ be fixed, with $p > M_{\ell, r}$. Then
$$n_{s,p}(r) = 2\sum_{d \mid r} \mu(d) \sum_{\Delta \in d\left(\mathcal{I}_{\ell,n \mid \frac{r}{d}}\right)} \delta_p(\Delta)h_2(\Delta),$$ where $h_2(\Delta) = |\cl(\mathcal{O}_\Delta)[2]|$, and $\delta_p(\Delta) = \begin{cases}1 & \parbox[center]{12em}{$p$ is inert in $\mathcal{O}$ and \\ $H_\mathcal{O}(x)$ has a solution in $\FF_p$,} \\ 0 & \text{otherwise.}\end{cases}$
\end{theorem}

\begin{remark}\label{ExplicitFormulaIsElementaryRmk}
Note that using Theorem \ref{ChenXueRootCountThm} and Theorem \ref{FundGenusTheoryThm}, $\delta_p(\Delta)$ and $h_2(\Delta)$ can both be calculated using only the factorization of $\Delta$, legendre symbol computations, and modular arithmetic. Thus the formula in Theorem \ref{ExplicitNsCountThm} depends entirely on ``elementary'' techniques, and does not require knowledge of the relevant class groups.
\end{remark}

\begin{proof}
By Theorem \ref{CycleBijectionThm}, we have a map from isogeny cycles  of length $r$ in $\mathcal{G}_\ell$ to orders in $\mathcal{I}_{\ell, r}$, given by finding the unique order $\mathcal{O}$ such that the cycle appears as a rim in $\mathcal{G}_{\mathcal{O}, \ell}$. Restricting to cycles along the spine, Corollary \ref{FpVertexOrientationSymmetryCorollary} shows that if $\mathcal{O}$ produces spine cycles, then it produces $2h_2(\Delta)$ of them, with the factor of $2$ coming from considering directions of traversal. We therefore have that $\sum_{d \mid r} n_{s,p}(d) = 2\sum_{\Delta \in d\left(\mathcal{I}_{\ell, n \mid d}\right)} \delta_p(\Delta)h_2(\Delta)$, and M\"obius inversion finishes the proof.
\end{proof}

Turning to the average number of spine cycles as $p$ varies, we have the following corollary:

\begin{corollary}\label{AverageNsCor}
Let $\{p_1, \hdots, p_n\}$ be a set of primes, all of which are greater than or equal to $M_{\ell,r}$. Then we have that 
\begin{equation}\label{EquationAverageNsCor}
\frac{1}{n}\sum_{i=1}^n n_{s,p_i}(r) = 2\sum_{d \mid r} \mu(d) \sum_{\Delta \in d(\mathcal{I}_{\ell, n \mid \frac{r}{d}})} \left(\frac{1}{n}\sum_{i=1}^n \delta_{p_i}(\Delta)\right)h_2(\Delta).
\end{equation}
\end{corollary}

\begin{proof}
This follows from Theorem \ref{EventualDistributionThm} by collecting all terms with the same $d, \Delta$.
\end{proof}

In order to compute the product on the right hand side of \eqref{EquationAverageNsCor}, we turn our attention to computing the average of the $\delta_{p_i}(\Delta)$. \\

Computing this average from Theorem \ref{ChenXueRootCountThm} is possible via a case-by-case analysis, but we will give a more conceptual proof based on the work of Li, Li, and Ouyang. We begin with the following theorem:

\begin{theorem}[\cite{LiLiOuyang21}, Theorem 4.1]\label{LiLiOuyangThm}
Let $\mathcal{O}$ be an order in an imaginary quadratic field $K$, and $\mathcal{H}_\mathcal{O}(x)$ be its Hilbert class polynomial. Let $F$ be the genus field of $\mathcal{O}$, $F^+ = F \cap \mathbb{R}$, and $p$ be a prime that does not split in $K$. Further, assume that $p \nmid \disc \mathcal{H}_\mathcal{O}(x)$ and $p \nmid \disc(\mathcal{O})$. Then $\mathcal{H}_\mathcal{O}(x)$ has a root in $\FF_p$ if and only if $p$ splits completely in $F^+$, 
\end{theorem}

In the cases of interest to us, the requirements that $p \nmid \disc \mathcal{H}_\mathcal{O}(x)$ and $p \nmid \disc(\mathcal{O})$ will always be satisfied, as the following lemma shows:

\begin{lemma}\label{LiLiOuyangApplicabilityLemma}
If $p > M_{\ell, r}$, then $p \nmid \disc \mathcal{H}_\mathcal{O}(x)$ and $p \nmid \disc(\mathcal{O})$ for all $\mathcal{O} \in \mathcal{I}_{\ell, r}$.
\end{lemma}

\begin{proof}
By definition of $M_{\ell, r}$, we have that $p > M_{\ell, r}$ implies that $p > \disc(\mathcal{O})$ for all $\mathcal{O} \in \mathcal{I}_{\ell, r}$. Thus $p \nmid \disc(\mathcal{O})$ for any $\mathcal{O} \in \mathcal{I}_{\ell, r}$. Further, Corollary \ref{ChenXueInjectivityCor} implies that $\mathcal{H}_{\mathcal{O}}(x)$ has no repeated roots modulo $p$, and so we see that $p \nmid \disc \mathcal{H}_{D}(x)$ as well.
\end{proof}

We now use Theorem \ref{LiLiOuyangThm} to reinterpret $\delta_{p_i}(\Delta)$ in terms of the Artin symbol $\left(\frac{F/\QQ}{p}\right)$:

\begin{lemma}\label{ArtinSymbolLemma}
Let $\Delta \in d(\mathcal{I}_{\ell, r})$, and $p$ be a prime such that $p > M_{\ell, r}$. Let $\sigma \in \Gal(F / \QQ)$ be the unique Galois element such that $\sigma|_{F^+}$ is the identity, and $\sigma|_{K}$ is conjugation. Then $\delta_p(\Delta) = 1$ if and only if $\left(\frac{F/\QQ}{p}\right) = \sigma$. 
\end{lemma}

\begin{proof}
Since $p > M_{\ell, r}$ we have that $p > \Delta$, and so $p$ is unramified in $K = \operatorname{Frac}(\mathcal{O}_\Delta)$. Since $F$ is contained in the Hilbert class field of $K$, which is an unramified extension of $K$, we have that $p$ is unramified in $F$ as well. Thus $\left(\frac{F/\QQ}{p}\right)$ is well-defined. \\

Next, we note that $\sigma$ is well-defined, since $F^+ \subset \RR$, and $K$ is imaginary quadratic, so we have that $F^+ \cap K = \QQ$. Thus $[F^+K : \QQ] = [K : \QQ][F^+ : \QQ] =  2(2^{\mu - 1})$. This is equal to the degree of $F$ over $\QQ$, and since $F^+ K \subset F$, we have that $F = F^+ K$. It follows that $\Gal(F / \QQ) \cong \Gal(F^+ / \QQ) \times \Gal(K / \QQ)$, which shows that $\sigma$ is well-defined. \\

Finally, we come to the proof that $\delta_p(\Delta) = 1$ if and only if $\left(\frac{F / \QQ}{p}\right) = \sigma$. By Theorem \ref{LiLiOuyangThm} and Lemma \ref{LiLiOuyangApplicabilityLemma}, $\delta_p(\Delta) = 1$ if and only if $p$ is inert in $K$ and splits completely in $F^+$. Since $p$ is unramified in $F^+$, the second requirement is equivalent to requiring that the inertia degree of $p$ is $1$, which is in turn equivalent to the assertion that $\left(\frac{F/\QQ}{p}\right) = 1$ by \cite[Corollary 5.21]{Cox}. Similarly, the first condition is equivalent to the inertia degree of $p$ in $K$ being $2$, which is equivalent to $\left(\frac{K / \QQ}{p}\right)$ being conjugation since this is the only element of order $2$ in $\Gal(K / \QQ)$. \\

Finally, by \cite[Exercise 5.16]{Cox}, we have that $\left(\frac{F^+ / \QQ}{p}\right)$ is the restriction of $\left(\frac{F / \QQ}{p}\right)$ to $F^+$, and similarly for $K$.
\end{proof}

We now use Chebotarev density to determine $\frac{1}{n} \sum_{i=1}^n \delta_{p_i}(\Delta)$. Recall the following version of Chebotarev density:

\begin{theorem}[\cite{LangANT}, pg. 167]\label{TcebotarevDensityThm}
Let $K / \QQ$ be an abelian extension, and let $\sigma \in \Gal(K / \QQ)$. Then 
$$\lim_{x \to \infty} \frac{\#\{p : p \leq x, p \text{ unramified, } \left(\frac{K / \QQ}{p}\right) = \sigma\}}{\#\{p : p \leq x\}} = \frac{1}{[K : \QQ]},$$ where $\left(\frac{K / \QQ}{p}\right)$ is the Artin symbol.
\end{theorem}

We have the following lemma:

\begin{lemma}\label{AverageOfDeltaPiLemma}
Let $(p_i)_{i=1}^\infty$ be an increasing sequence of consecutive primes, with $p_1 > M_{\ell, r}$, and $\Delta \in d(\mathcal{I}_{\ell,r})$. Then $$\lim_{n \to \infty} \frac{1}{n}\sum_{i=1}^n \delta_{p_i}(\Delta) = \frac{1}{2^\mu}.$$
\end{lemma}

\begin{proof}
Let $\pi_\sigma(x) = \#\{p : p \leq x, p \text{ unramified, and } \delta_p(\Delta) = 1\}$ and $\pi(x) = \#\{ p : p \leq x\}$. By Lemma \ref{ArtinSymbolLemma} and Theorem \ref{TcebotarevDensityThm}, we have $$\lim_{x \to \infty} \frac{\pi_\sigma(x)}{\pi(x)} = \frac{1}{[F : \QQ]} = \frac{1}{2^{\mu}}.$$ \\

We now show that $$\lim_{n \to \infty} \frac{\pi_\sigma(p_n)}{\pi(p_n)} = \frac{1}{n}\sum_{i=1}^n \delta_{p_i}(\Delta).$$ Extend $\delta_p(\Delta)$ to all primes $p$ by setting $\delta_p(\Delta) = \begin{cases}1 & \text{if $\left(\frac{K / \QQ}{p}\right) = \sigma$} \\ 0 & \text{otherwise}\end{cases}$. Then we can compute directly that
\begin{align*}
\frac{\pi_\sigma(p_n)}{\pi(p_n)} & = \frac{\sum_{i=1}^n \delta_{p_i}(\Delta) + \sum_{p < p_1} \delta_p(\Delta)}{n + \pi(p_1)} \\
& = \left(1 + \frac{\pi(p_1)}{n}\right)^{-1}\left(\frac{1}{n}\sum_{i=1}^n \delta_{p_i}(\Delta)\right) + \left(1 + \frac{\pi(p_1)}{n}\right)^{-1}\left(\frac{1}{n}\sum_{p < p_1}\delta_p(\Delta)\right),
\end{align*}
and the result follows by taking $n \to \infty$.
\end{proof}

Combining Lemma \ref{AverageOfDeltaPiLemma} and Corollary \ref{AverageNsCor}, we can prove the final result from the introduction:

\begin{theorem}\label{AverageNsThm}
Let $(p_i)_{i=1}^\infty$ be an increasing sequence of consecutive primes. Then 
$$\lim_{n\to \infty} \frac{1}{n}\sum_{i=1}^n n_{s,p_i}(r) = \sum_{d \mid r} \mu(d) \#d(\mathcal{I}_{\ell, n \mid \frac{r}{d}}).$$
\end{theorem}

\begin{proof}
We first note that since the sequence $(p_i)$ is increasing, only finitely many of the terms are less than $M_{\ell, r}$. Thus Corollary \ref{AverageNsCor} applies unchanged. \\

Next, note that the first two sums in Corollary \ref{AverageNsCor} are finite, we can interchange the limit and the sums to get that $$\lim_{n\to \infty} \frac{1}{n}\sum_{i=1}^n n_{s,p_i}(r) = \sum_{d \mid r} \mu(d) \sum_{\Delta \in d\left(\mathcal{I}_{\ell, n \mid \frac{r}{d}}\right)} 2 h_2(\Delta)\lim_{n \to \infty} \left(\frac{1}{n}\sum_{i=1}^n \delta_{p_i}(\Delta)\right).$$ The theorem follows from the fact that $h_2(\Delta) = 2^{\mu - 1}$, and the remaining limit is $\frac{1}{2^\mu}$ by Lemma \ref{AverageOfDeltaPiLemma}.
\end{proof}

As an example, we compute the expected number of $3$-cycles along the spine in $\mathcal{G}_3$.

\begin{example}\label{AvgNsThmExample}
Using Proposition \ref{DiscriminantDescriptionProp}, we see that $$d(\mathcal{I}_{3, n \mid 3}) = \{-107,-104,-92,-83,-59,-44,-23,-11,-8\},$$ and $d(\mathcal{I}_{3, n \mid 1}) = \{-11, -8\}$. Thus the expected number of $3$-cycles along the spine in $\mathcal{G}_3$ is given by 
$$\mu(1)\#d(\mathcal{I}_{3, n \mid 3}) + \mu(3)\#d(\mathcal{I}_{3, n \mid 1}) = 9 - 2 = 7.$$
\end{example} 

Example \ref{AvgNsThmExample} demonstrates the following corollary:

\begin{corollary}\label{AvgNsCountPrimeOrderCor}
With the notation and assumptions from Theorem \ref{AverageNsThm}, suppose further that $r$ is prime. Then $$\lim_{n \to \infty} \frac{1}{n} \sum_{i=1}^n n_{s,p_i}(r) = \#d(\mathcal{I}_{\ell, r}).$$
\end{corollary}

\begin{proof}
Since $r$ is prime, the only divisors of $r$ are $1$ and $r$. By Theorem \ref{AverageNsThm}, we have that 
\begin{align*}
\lim_{n \to \infty} \frac{1}{n} \sum_{i=1}^n n_{s,r}(p_i) & = \mu(1)\#d(\mathcal{I}_{\ell, n \mid r}) + \mu(r) \#d(\mathcal{I}_{\ell, n \mid 1}) \\
& = \mu(1)(\#d(\mathcal{I}_{\ell, r}) + \#d(\mathcal{I}_{\ell, 1})) + \mu(r) \#d(\mathcal{I}_{\ell, 1}) \\
& = \#d(\mathcal{I}_{\ell, r}) + \#d(\mathcal{I}_{\ell, 1}) - \#d(\mathcal{I}_{\ell, 1}) \\
& = \#d(\mathcal{I}_{\ell, r}).
\end{align*}
\end{proof}

\subsection{Remarks on computational verification}\label{ComputationalVerificationSubsection}

To verify the results of Theorem \ref{AverageNsThm} and Corollary \ref{AverageNsCor}, we collected data on $3$-cycles in the $3$-isogeny and $5$-isogeny graphs for all primes up to $10,000$. \\

\begin{figure}

\centering
\subfloat[3-cycles in the $3$-isogeny graph]{\label{Ell3Fig}
\centering
\includegraphics[width=0.46\linewidth]{"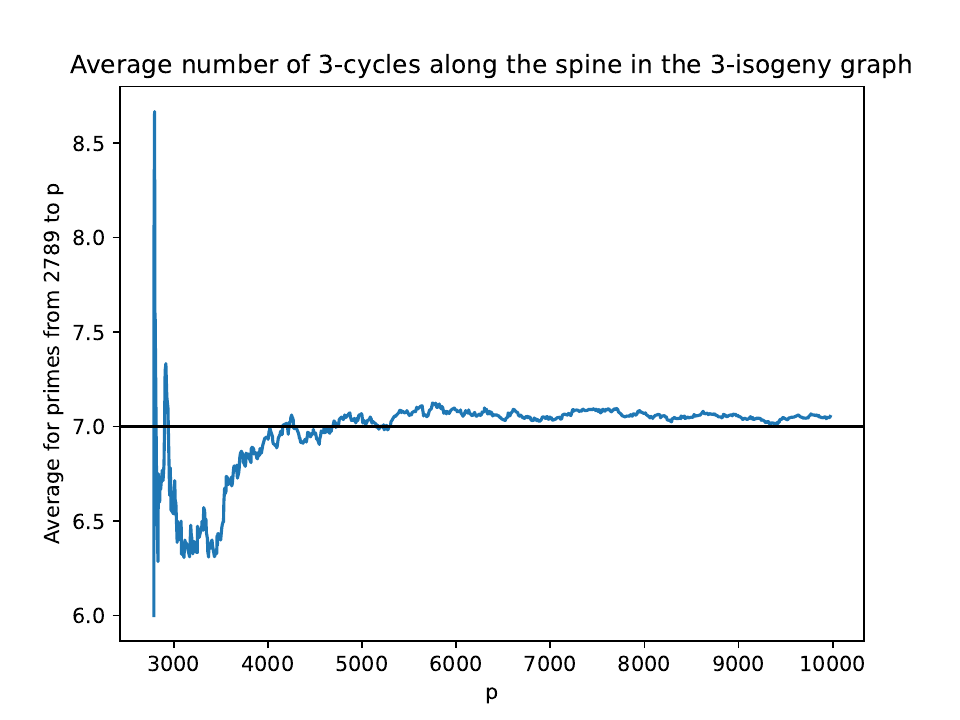"}
}
\hfill
\subfloat[3-cycles in the $5$-isogeny graph]{\label{Ell5Fig}
\centering

\includegraphics[width=0.46\linewidth]{"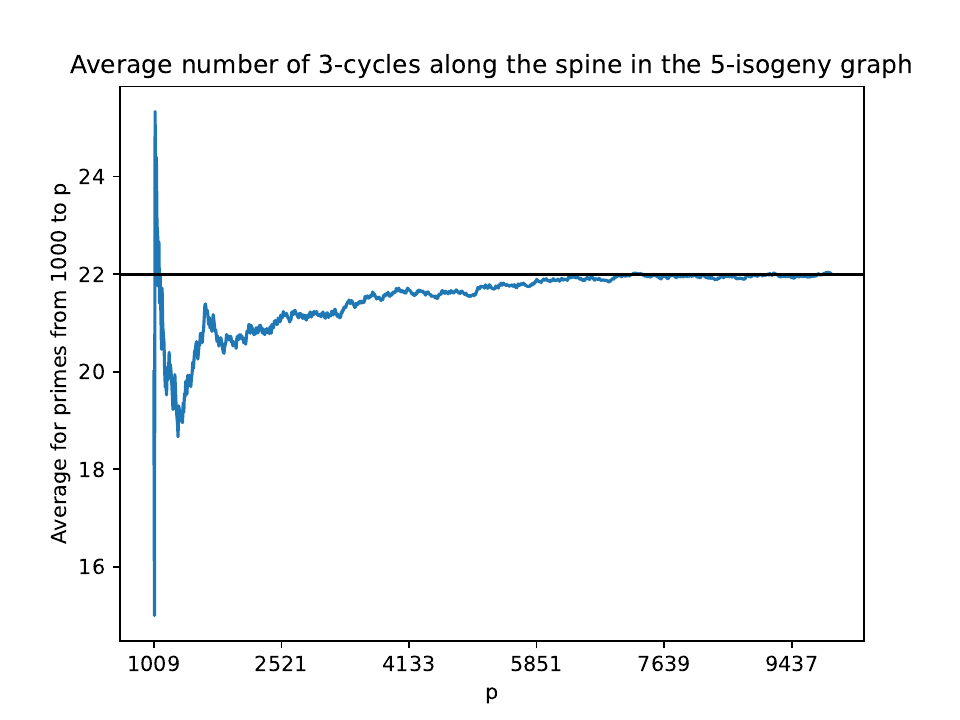"}
}

\caption{Average numbers of $3$-cycles along the spine in the $3$-isogeny and $5$-isogeny graphs}
\end{figure}

Figures \ref{Ell3Fig} and \ref{Ell5Fig} show the results, with the black lines indicating the predicted limit in Theorem \ref{AverageNsCor}. In the case of Figure \ref{Ell3Fig}, $2,789$ was chosen because it is the first prime larger than $M_{3,3}$. Since $M_{5,3} = 61,876$, considering only primes larger than $M_{5,3}$ was impractical, so $1,000$ was chosen arbitrarily. In particular, we see that even at primes less than $M_{\ell, r}$, the average converges quickly toward the final limit. \\

In Figure \ref{Ell3Fig}, we see that the average remains slightly above the predicted limit even at the largest primes. We hypothesize that this happens due to inequalities in prime races. The number of $3$-cycles along the spine produced by each determinant in $\mathcal{I}_{\ell, r}$ is determined by congruence conditions on $p$ modulo $d$ or $4d$, so it is likely that certain congruences are more common for small values of $p$. If these congruences produce more or less than the average number of $3$-cycles along the spine, then this would cause the average number of $3$-cycles along the spine for primes up to large bounds to remain slightly above or below the limit.

\section{Cycles of even length}\label{SecEvenCase}

We now consider the case of even cycles, where we are able to obtain partial results. The only place where our assumption that $r$ has odd length was used in an essential fashion is in Proposition \ref{FpvertexOrientationSymmetryOddLengthProp}, when we proved that each cycle of length $r$ contains either $0$ or $1$ $\FF_p$-vertices. \\

In this argument we showed that for large enough $p$, the number of $\FF_p$-vertices on each cycle in $\mathcal{G}_{\mathcal{O}, \ell}$ of length $r$ is either $0$ or odd. Further, this number has to divide a power of $2$. This allowed us to conclude that any cycle of odd length $r$ in $\mathcal{G}_\ell$ with one $\FF_p$-vertex has exactly one $\FF_p$-vertex. In the case where $r$ is even, however, the number of $\FF_p$-vertices on each cycle in $\mathcal{G}_{\mathcal{O}, \ell}$ can be a non-trivial power of $2$. In this section, we alter the graph theoretic arguments to show that under some additional assumptions, the number of such vertices is exactly $2$. We begin by defining a modified version of our bound $M_{\ell, r}$:

\begin{definition}\label{EvenBoundDef}
Let $r$ be even. We set $M_{\ell, r}^{\mathrm{strong}} := \max \{M_{\ell, r_i} : r_i < r\}.$
\end{definition}

\begin{remark}\label{EvenBoundRmk}
We expect that $M_{\ell, r}^{\textrm{strong}}$ is always equal to $M_{\ell, r}$, since the discriminants where primes above $\ell$ have order $r$ in the class group will generally be larger than discriminants where they have order $r_i$ for $r_i < r$. For the purposes of this paper, however, it is most important to have a finite bound. 
\end{remark}

We can now state partial results for Theorem \ref{EventualDistributionThm} in the case where $r$ is even:

\begin{theorem}\label{EvenCyclesThm}
Let $r$ be an even number that is not a power of $2$, and $p > M_{\ell, r}^{\textrm{strong}}$.
Then the results of Theorem \ref{EventualDistributionThm} and Corollary \ref{EventualDistributionCor} hold unchanged.
\end{theorem}   

The proof strategy will be the same, except that we will replace the result of Proposition \ref{FpvertexOrientationSymmetryOddLengthProp} with Proposition \ref{EvenFpVertexDistProp} below. This proposition says that instead of $0$ or $1$ $\FF_p$-vertex, each cycle will contain $0$ or $2$ $\FF_p$-vertices. Thus, we will get the following modified version of Theorem \ref{AverageNsThm}:

\begin{theorem}\label{EvenAvgNsThm}
Let $(p_i)_{i=1}^\infty$ be an increasing sequence of consecutive primes, and $r$ be even. Then if $r$ is not a power of $2$ we have 
$$\lim_{n\to \infty} \frac{1}{n}\sum_{i=1}^n n_{s,p_i}(r) = \frac{1}{2}\sum_{d \mid r} \mu(d) \#d(\mathcal{I}_{\ell, n \mid \frac{r}{d}}).$$
\end{theorem}

The only modification of the proofs of Theorem \ref{EventualDistributionThm} and Theorem \ref{AverageNsThm} needed to prove Theorem \ref{EvenCyclesThm} and Theorem \ref{EvenAvgNsThm} is replacing Proposition \ref{FpvertexOrientationSymmetryOddLengthProp} with the following proposition:

\begin{proposition}\label{EvenFpVertexDistProp}
Let $r$ be even and $p > M_{\ell, r}^{\mathrm{strong}}$. If $r$ is not a power of $2$, and $C$ is an $r$-cycle in $\mathcal{G}_\ell$, then the number of $\FF_p$-vertices on $C$ is either $0$ or $2$.
\end{proposition}

Before proving Proposition \ref{EvenFpVertexDistProp}, we prove the following Lemma:

\begin{lemma}\label{EvenFpVertexDistPropLemma}
Let $p > M_{\ell, r}^{\mathrm{strong}}$, and $C$ be an isogeny cycle in $\mathcal{G}_\ell$ of length $r$. Let $v$ be a vertex in $C$. Then there is only one incoming and one outgoing edge connected to $v$ in $C$. Further, if $C'$ is the cycle in $\mathcal{G}_{\mathcal{O}, \ell}$ corresponding to $C$ via Theorem \ref{CycleBijectionThm}, then the map from $C'$ to $C$ obtained by forgetting orientations is injective on vertices. 
\end{lemma}

\begin{proof}
Under the bijection of Theorem \ref{CycleBijectionThm}, there is an order $\mathcal{O} \in \mathcal{I}_{\ell, r}$, and a cycle $C'$ in $\mathcal{G}_{\mathcal{O},  \ell}$ that corresponds to $C$ after forgetting orientations. The conclusion of the Lemma holds for $C'$ by Proposition \ref{OrientedVolcanoeProp}, and forgetting orientations preserves sources and targets of edges, so we need only show that forgetting orientations is injective on vertices. \\

Suppose instead that there are two vertices in $C'$ with the same $j$-invariant. Let $j$ be this $j$-invariant. We then obtain two cycles in $\mathcal{G}_\ell$ through $j$ upon forgetting orientations.  Both of these cycles have length at most $r$. By composing the isogenies in these cycles, we obtain two cyclic endomorphisms in $\End(E(j))$, of degree $\ell^{r_i}, \ell^{r_j}$ for $r_i, r_j < r$. These endomorphisms then give embeddings of imaginary quadratic orders in $\mathcal{I}_{\ell, r_i}$ and $\mathcal{I}_{\ell, r_j}$ into $\End(E(j))$. This contradicts Remark \ref{KanekoLemmaRmk}, because we have chosen $p > M_{\ell, r}^{\mathrm{strong}}$. Thus we see that forgetting orientations is injective on vertices.
\end{proof}

We now return to the proof of Proposition \ref{EvenFpVertexDistProp}.

\begin{proof}[Proof (Proposition \ref{EvenFpVertexDistProp})]
By Corollary \ref{FpVertexOrientationSymmetryCorollary}, there is a constant $K$ such that each $r$-cycle in $\mathcal{G}_{\mathcal{O}, \ell}$ containing an $\FF_p$-vertex contains $K$ $\FF_p$-vertices. Since the total number of such vertices is a power of $2$, the number of vertices on each cycle cannot be $r$, since $r$ is not a power of $2$. By Lemma \ref{EvenFpVertexDistPropLemma}, each $r$-cycle in $\mathcal{G}_{\ell}$ also contains $r$ vertices, and the number of $\FF_p$-vertices is again a power of $2$. Thus each $r$-cycle in $\mathcal{G}_\ell$ contains $\FF_{p^2}$-vertices. Let $C$ be such a cycle containing at least one $\FF_p$-vertex. We will show that $C$ contains exactly $2$ $\FF_p$-vertices. \\

Let $C$ be given by a sequence of isogenies $\phi_1, \phi_2, \hdots, \phi_r$. Since we are forgetting basepoint in Definition \ref{IsogenyCycleDef}, we may assume that the domain of $\phi_1$ is an $\FF_p$-vertex, and the codomain is an $\FF_{p^2}$-vertex. Let $i$ be the first index such that the codomain of $\phi_i$ is an $\FF_p$-vertex, and let $v$ be this vertex. We will show that $i = \frac{r}{2}$, from which it will follow that there are exactly two $\FF_p$-vertices in $C$. \\

By changing basepoint to $v$ if necessary, we may assume that $i \leq \frac{r}{2}$. Applying Frobenius to the sequence of isogenies $\phi_1, \phi_2, \hdots, \phi_i$, we obtain a sequence $\phi_1^{(p)}, \phi_2^{(p)}, \hdots, \phi_i^{(p)}$, which begins at the domain of $\phi_1$, and ends at $v$. The vertices in between in this sequence are distinct from the vertices traversed in the sequence $\phi_1, \phi_2, \hdots, \phi_i$ by the assumption that $i$ is the first index where the codomain of $\phi_i$ is defined over $\FF_p$. Taking duals therefore gives a closed walk $\phi_1, \hdots, \phi_i, \hat{\phi}_i^{(p)}, \hdots, \hat{\phi}_1^{(p)}$ with no-backtracking through $v$. This is an isogeny cycle of length $2i \leq r$, which we denote by $C'$. By Remark \ref{KanekoLemmaRmk}, our assumption that $p > M_{\ell, r}^{\mathrm{strong}}$ guarantees that this cycle is either $C$ or $\hat{C}$. Since $C$ and $\hat{C}$ have the same vertices, and there are exactly $2$ $\FF_p$-vertices in $C'$, we see that $C$ has exactly $2$ $\FF_p$-vertices. This finishes the proof. \\
\end{proof}

See Figures \ref{Ell3_evenconj_fig} and \ref{Ell2_evenconj_fig} for experimental data verifying the results of Theorem \ref{EvenAvgNsThm}. We note that Figure \ref{Ell3_evenconj_fig} provides computational evidence that the results of Theorem \ref{EvenCyclesThm} and Theorem \ref{EvenAvgNsThm} both hold even in the case where $r$ is a power of $2$. In this case however, the possibility of cycles through the spine that are completely fixed by Frobenius makes the graph theoretic arguments in Propostion \ref{EvenFpVertexDistProp} fail to provide a proof. \\

We end with a remark on an interesting consequence of Proposition \ref{EvenFpVertexDistProp}:

\begin{remark}\label{SpineVerticesStickTogetherRmk}
By taking $r = 2n$, Proposition \ref{EvenFpVertexDistProp} says that there are infinitely many primes $p$ where $\mathcal{G}_\ell$ contains spine vertices of distance at most $n$ from each other. The vertices on the spine are exactly the vertices whose endomorphism rings contain an embedded copy of $\ZZ[\sqrt{-p}]$. Boneh and Love studied the distribution of vertices with small endomorphisms in \cite{BonehLove20}, and found that small endomorphisms have a repelling effect. In this context, our result says that vertices with endomorphisms of degree roughly $\sqrt{p}$ can remain close together in $\mathcal{G}_\ell$.
\end{remark}

\begin{figure}

\centering
\subfloat[4-cycles in the $3$-isogeny graph]{\label{Ell3_evenconj_fig}
\centering
\includegraphics[width=0.46\linewidth]{"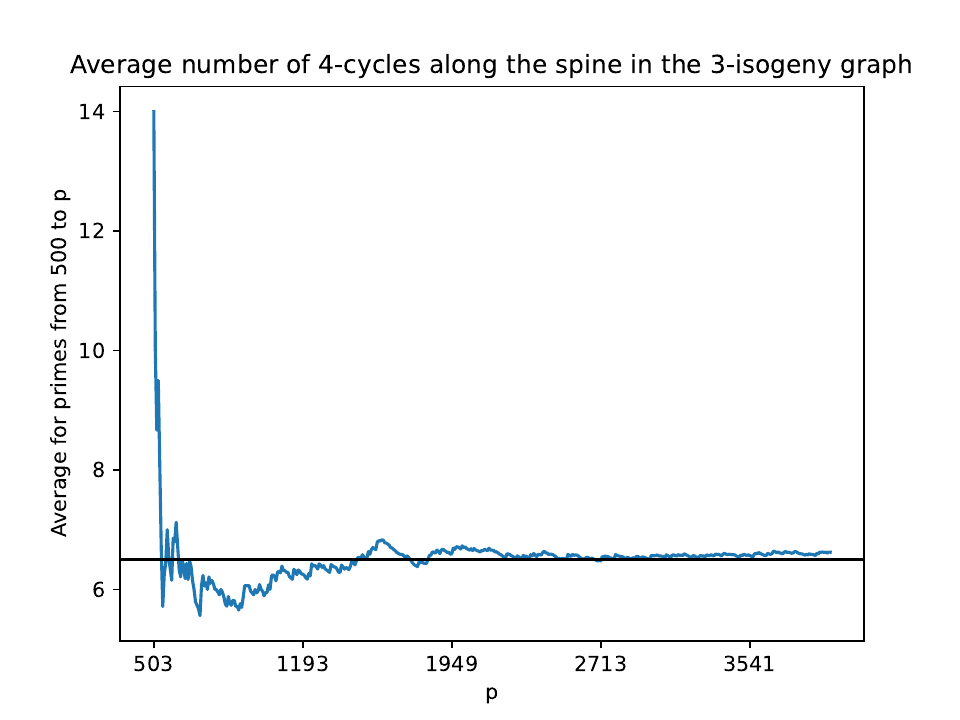"}
}
\hfill
\subfloat[6-cycles in the $2$-isogeny graph]{\label{Ell2_evenconj_fig}
\centering

\includegraphics[width=0.46\linewidth]{"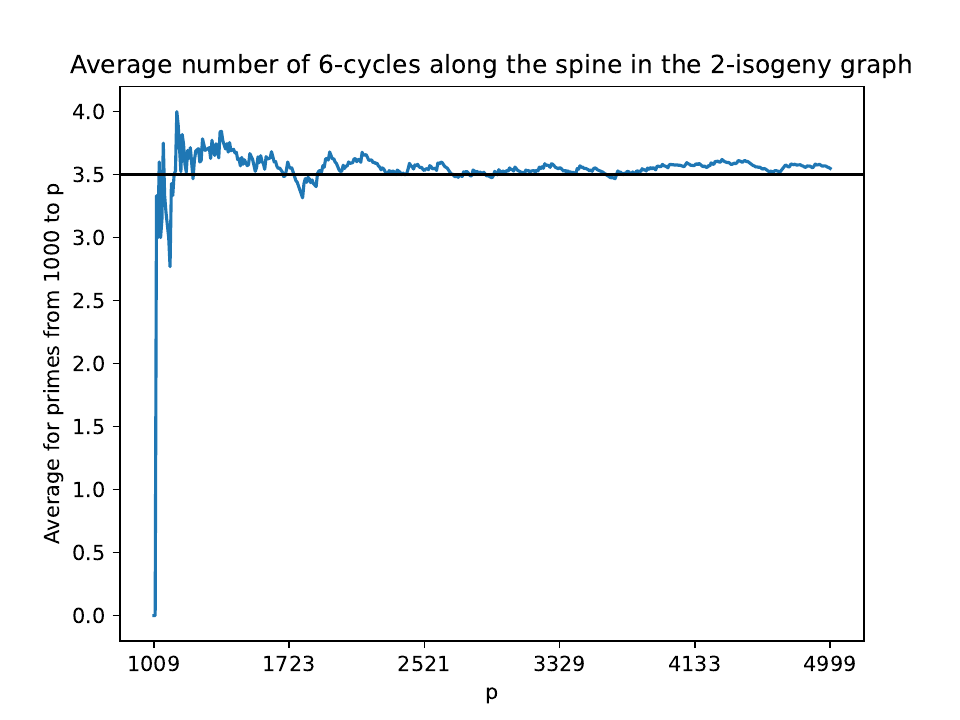"}
}

\caption{Average numbers of $4$-cycles (resp. $6$-cycles) along the spine in the $3$-isogeny (resp. $2$-isogeny) graph}
\end{figure}

\section*{}

\printbibliography
\end{document}